\documentclass[11pt]{article}

\usepackage{amssymb,amsmath,amsfonts,amsthm}
\usepackage{latexsym}
\usepackage{graphics}
\usepackage{indentfirst}

\newtheorem{innercustomthm}{Theorem}
\newenvironment{customthm}[1]
  {\renewcommand\theinnercustomthm{#1}\innercustomthm}
  {\endinnercustomthm}

\newtheorem{innercustomcor}{Corollary}
\newenvironment{customcor}[1]
  {\renewcommand\theinnercustomcor{#1}\innercustomcor}
  {\endinnercustomcor}

\newtheorem*{thm*}{Theorem}
\newtheorem*{pro*}{Proposition}
\newtheorem{thm}{Theorem}

\newtheorem{pro}[thm]{Proposition}
\newtheorem{obs}[thm]{Observation}
\newtheorem{cor}[thm]{Corollary}

\newcommand{\N}{\mathbb{N}}
\newcommand{\Z}{\mathbb{Z}}

\newcommand{\R}{\mathbb{R}}

\newcommand{\diff}{\mathrm{diff}}

\newcommand{\col}{\mathrm{col}}

\begin{document}

\title{Combinatorial Nullstellensatz and DP-coloring of Graphs}

\author{Hemanshu Kaul\footnote{Department of Applied Mathematics, Illinois Institute of Technology, Chicago, IL 60616. E-mail: {\tt kaul@iit.edu}} \\
Jeffrey A. Mudrock\footnote{Department of Mathematics, College of Lake County, Grayslake, IL 60030. E-mail: {\tt jmudrock@clcillinois.edu}}  }

\maketitle

\begin{abstract}

\medskip

We initiate the study of applying the Combinatorial Nullstellensatz to the DP-coloring of graphs even though, as is well-known, the Alon-Tarsi theorem does not apply to DP-coloring. We define the notion of \emph{good covers of prime order} which allows us to apply the Combinatorial Nullstellensatz to DP-coloring. We apply these tools to DP-coloring of the cones of certain bipartite graphs and uniquely 3-colorable graphs. We also extend a result of Akbari, Mirrokni, and Sadjad (2006) on unique list colorability to the context of DP-coloring. We establish a sufficient algebraic condition for a graph $G$ to satisfy $\chi_{DP}(G) \leq 3$, and we completely determine the DP-chromatic number of squares of all cycles.



\end{abstract}

\section{Introduction}\label{intro}

We only consider nonempty, finite, undirected, simple graphs unless otherwise noted. Generally speaking we follow West~\cite{W01} for terminology and notation.  The set of natural numbers is $\N = \{1,2,3, \ldots \}$.  Given a set $A$, $\mathcal{P}(A)$ is the power set of $A$.  For $m \in \N$, we write $[m]$ for the set $\{1, \ldots, m \}$.  If $G$ is a graph and $S, U \subseteq V(G)$, we use $G[S]$ for the subgraph of $G$ induced by $S$, and we use $E_G(S, U)$ for the subset of $E(G)$ with an endpoint in $S$ and an endpoint in $U$.  If an edge in $E(G)$ connects the vertices $u$ and $v$, the edge can be represented by $uv$ or $vu$.  We use $\alpha(G)$ for the size of the largest independent set in $G$.  For $v \in V(G)$, we write $d_G(v)$ for the degree of vertex $v$ in the graph $G$, and we use $\Delta(G)$ and $\delta(G)$ for the maximum and minimum degree of a vertex in $G$ respectively.  If $G_1$ and $G_2$ are vertex disjoint graphs, we write $G_1 \vee G_2$ for the join of $G_1$ and $G_2$. Also, $G^k$ denotes the $k^{th}$ power of graph $G$ (i.e., $G^k$ has the same vertex set as $G$ and edges between any two vertices within distance $k$ in $G$).

If $D$ is a simple directed graph (i.e., digraph), we view the edge set of $D$ as a subset of $V(D) \times V(D)$.  When $(u,v) \in E(D)$ we say that $u$ is the \emph{tail} of the edge and $v$ is the \emph{head} of the edge.  For $v \in V(D)$ we use $d_G^{+}(v)$ for the outdegree of $v$, and we use $d_G^{-}(v)$ for the indegree of $v$.  An orientation of an undirected graph $G$ is a directed graph $D$ with $V(D) = V(G)$ and each edge in $E(D)$ is formed as follows: for each $uv \in E(G)$, exactly one of the ordered pairs: $(u,v)$ or $(v,u)$ is put in $E(D)$.

\subsection{Combinatorial Nullstellensatz and its Application to List Coloring}

The focus of this paper is to apply the Combinatorial Nullstellensatz to DP-coloring of graphs, a far-reaching and insightful generalization of list coloring of graphs. In this subsection we review the Combinatorial Nullstellensatz and briefly overview how it has been applied to list coloring problems.

Suppose that $\mathbb{F}$ is a field.  We use $\mathbb{F}[x_1, \ldots, x_n]$ for the ring of polynomials in $n$ variables with coefficients in $\mathbb{F}$.  If $p(x_1, \ldots, x_n) \in \mathbb{F}[x_1, \ldots, x_n]$, then we use $[ \prod_{i=1}^n x_i^{t_i}]_p $ to denote the element in $\mathbb{F}$ that is the coefficient of the term $\prod_{i=1}^n x_i^{t_i}$ in the expanded form of $p(x_1, \ldots, x_n)$.  We are now ready to state the Combinatorial Nullstellensatz.

\begin{thm} [Combinatorial Nullstellensatz~\cite{A99}] \label{thm: combnull}
Suppose $\mathbb{F}$ is a field and $f \in \mathbb{F}[x_1, \ldots, x_n]$.  If $f$ has degree $\sum_{i=1}^n t_i$, $[\prod_{i=1}^n x_i^{t_i}]_f \neq 0$, and $S_1, \ldots, S_n$ are subsets of $\mathbb{F}$ with $|S_i| > t_i$ for each $i \in [n]$, then there is an $(s_1, \ldots, s_n) \in \prod_{i=1}^n S_i$ such that $f(s_1, \ldots, s_n) \neq 0$ \footnote{Note that when we evaluate $f(s_1, \ldots, s_n)$, addition and multiplication are performed in $\mathbb{F}$.}.
\end{thm}

Many exciting applications of Theorem~\ref{thm: combnull} have appeared in the literature.  Since we are seeking to apply Theorem~\ref{thm: combnull} to a generalization of list coloring, we will now highlight how Theorem~\ref{thm: combnull} has been applied to list coloring.  We will begin with some definitions.

In the classical vertex coloring problem we wish to color the vertices of a graph $G$ with up to $m$ colors from $[m]$ so that adjacent vertices receive different colors, a so-called \emph{proper $m$-coloring}. The chromatic number of a graph, denoted $\chi(G)$, is the smallest $m$ such that $G$ has a proper $m$-coloring.  List coloring is a well-known variation on the classical vertex coloring problem, and it was introduced independently by Vizing~\cite{V76} and Erd\H{o}s, Rubin, and Taylor~\cite{ET79} in the 1970's.  For list coloring, we associate a \emph{list assignment}, $L$, with a graph $G$ such that each vertex $v \in V(G)$ is assigned a list of colors $L(v)$ (we say $L$ is a list assignment for $G$).  The graph $G$ is \emph{$L$-colorable} if there exists a proper coloring $f$ of $G$ such that $f(v) \in L(v)$ for each $v \in V(G)$ (we refer to $f$ as a \emph{proper $L$-coloring} of $G$).

A list assignment $L$ is called an \emph{$m$-assignment} for $G$ if $|L(v)|=m$ for each $v \in V(G)$.  The \emph{list chromatic number} of a graph $G$, denoted $\chi_\ell(G)$, is the smallest $m$ such that $G$ is $L$-colorable whenever $L$ is an $m$-assignment for $G$ \footnote{Clearly, $\chi(G) \leq \chi_\ell(G)$ for any graph $G$ since an $m$-assignment could assign the same $m$ colors to every vertex in $V(G)$.}. We say $G$ is \emph{$m$-choosable} if $m \geq \chi_\ell(G)$.  If $f: V(G) \rightarrow \N$, we say that $G$ is \emph{$f$-choosable} if $G$ is $L$-colorable whenever $L$ is a list assignment for $G$ satisfying $|L(v)|=f(v)$ for each $v \in V(G)$.

Given a graph $G$ with $V(G) = \{v_1, \ldots, v_n \}$, the \emph{graph polynomial of $G$}, denoted $f_G$, is given by $f_G(x_1, x_2, \ldots, x_n) = \prod_{v_iv_j \in E(G), \; j>i} (x_i-x_j).$  Importantly, notice that it is possible to view $f_G$ as an element of $\mathbb{F}[x_1, \ldots, x_n]$ where $\mathbb{F}$ is any field~\footnote{Unless otherwise noted, when we are working with a graph polynomial $f_G(x_1, \ldots, x_n)$ we assume that $f_G \in \R[x_1, \ldots, x_n]$.}.  Notice also that $f_G$ is homogenous of degree $|E(G)|$.  Also, if $L$ is a list assignment for $G$ such that $L(v) \subset \N$ for each $v \in V(G)$, then $G$ is $L$-colorable if and only if there is a $(c_1, \ldots, c_n) \in \prod_{i=1}^n L(v_i)$ such that $f_G(c_1, c_2, \ldots, c_n) \neq 0$.  This observation makes it clear as to why Theorem~\ref{thm: combnull} can be useful in attacking list coloring problems.

We will now introduce one of the most useful tools for applying the Combinatorial Nullstellensatz to list coloring: the celebrated Alon-Tarsi Theorem.  Suppose $D$ is a simple digraph. We say that $E$ is a \emph{circulation} contained in $D$ if $E$ is a spanning subgraph of $D$ and for each $v \in V(D)$, $d^-_E(v) = d^+_E(v)$.  We say a circulation is \emph{even} (resp. \emph{odd}) if it has an even (resp. odd) number of edges.

Suppose $G$ is a graph with $V(G) = \{v_1, \ldots, v_n \}$, and suppose that $D$ is an orientation of $G$.  Let $\diff(D)$ denote the absolute value of the difference in the number of even and odd circulations contained in $D$.  Alon and Tarsi~\cite{AT92} showed that $\diff(D) = \left|[\prod_{i=1}^n x_i^{d^+_D(v_i)}]_{f_G} \right|$ when $f_G \in \R[x_1, \ldots, x_n]$.  This and Theorem~\ref{thm: combnull} immediately imply the following result.

\begin{thm} [Alon-Tarsi Theorem~\cite{AT92}] \label{thm: AT}
Let $D$ be an orientation of the graph $G$.  Suppose that $f:V(G) \rightarrow \N$ is given by $f(v) = 1 + d_D^+(v)$ for each $v \in V(G)$.  If the number of even and odd circulations contained in $D$ differ, then $G$ is $f$-choosable, and consequently $\chi_{\ell}(G) \leq \max_{v \in V(G)} (1 + d_D^+(v))$.
\end{thm}

Many applications and extensions of the Alon-Tarsi Theorem have appeared in the literature (see e.g.,~\cite{AT92, FS92, H11, JT95, KM18, PW03, S10, S12, S14, WQ19, Z17}).

\subsection{DP-coloring and generalized $S$-labelings of Graphs}

In 2015, Dvo\v{r}\'{a}k and Postle~\cite{DP15} introduced DP-coloring (they called it correspondence coloring) in order to prove that every planar graph without cycles of lengths 4 to 8 is 3-choosable. DP-coloring has been extensively studied over the past 5 years (see e.g.,~\cite{B16,B17, BK17, BK182,BK19,BK18, BK172, KM19, KO18, KO182, KY17, LL19, LLYY19, Mo18, M18}). Intuitively, DP-coloring considers the worst-case scenario of how many colors we have to use in a list coloring if we no longer can identify the names of the colors.  Following~\cite{BK17}, we now give the formal definition.  Suppose $G$ is a graph.  A \emph{cover} of $G$ is a pair $\mathcal{H} = (L,H)$ consisting of a graph $H$ and a function $L: V(G) \rightarrow \mathcal{P}(V(H))$ satisfying the following four requirements:

\vspace{5mm}

\noindent(1) the set $\{L(u) : u \in V(G) \}$ is a partition of $V(H)$; \\
(2) for every $u \in V(G)$, the graph $H[L(u)]$ is complete; \\
(3) if $E_H(L(u),L(v))$ is nonempty, then $u=v$ or $uv \in E(G)$; \\
(4) if $uv \in E(G)$, then $E_H(L(u),L(v))$ is a matching (the matching may be empty).

\vspace{5mm}

Suppose $\mathcal{H} = (L,H)$ is a cover of $G$. An \emph{$\mathcal{H}$-coloring} of $G$ is an independent set in $H$ of size $|V(G)|$.  It is immediately clear that an independent set $I \subseteq V(H)$ is an $\mathcal{H}$-coloring of $G$ if and only if $|I \cap L(u)|=1$ for each $u \in V(G)$.

Given a function $f: V(G) \rightarrow \N$, we say $\mathcal{H}$ is an \emph{$f$-cover} if $|L(u)|=f(u)$ for each $u \in V(G)$. Along the lines of $f$-choosability of graphs, we say that $G$ is \emph{$f$-DP-colorable} if $G$ is $\mathcal{H}$-colorable whenever $\mathcal{H}$ is an $f$-cover of $G$. We say $G$ is \emph{DP-$m$-colorable} if $G$ is $f$-DP-colorable when $f(u)=m$ for each $u \in V(G)$. A cover $\mathcal{H}=(L,H)$ is called \emph{$m$-fold} if $|L(u)|=m$ for each $u \in V(G)$.  The \emph{DP-chromatic number} of a graph $G$, $\chi_{DP}(G)$, is the smallest $m \in \N$ such that $G$ is DP-$m$-colorable.

Given an $m$-assignment, $L$, for a graph $G$, it is easy to construct an $m$-fold cover $\mathcal{H}$ of $G$ such that $G$ has an $\mathcal{H}$-coloring if and only if $G$ has a proper $L$-coloring (see~\cite{BK17}).  It follows that $\chi_\ell(G) \leq \chi_{DP}(G)$.  This inequality may be strict since it is easy to prove that $\chi_{DP}(C_n) = 3$ whenever $n \geq 3$, but the list chromatic number of any even cycle is 2 (see~\cite{BK17} and~\cite{ET79}).  Notice that like $m$-choosability, the graph property of having DP-chromatic number at most $m$ is monotone. The \emph{coloring number} of a graph $G$, denoted $\col(G)$, is the smallest integer $d$ for which there exists an ordering, $v_1, v_2, \ldots, v_n$, of the elements in $V(G)$ such that each vertex $v_i$ has at most $d-1$ neighbors among $v_1, v_2, \ldots, v_{i-1}$.  It is easy to prove that $\chi(G) \leq \chi_\ell(G) \leq \chi_{DP}(G) \leq \col(G) \leq \Delta(G)+1$, and Dvo\v{r}\'{a}k and Postle~\cite{DP15} observed that $\chi_{DP}(G) \leq \Delta(G)$ provided that $G$ is neither a cycle nor a complete graph.

Since there are striking similarities and differences between DP-coloring and list coloring (see the citations mentioned above), it is natural to ask whether the Combinatorial Nullstellensatz or an analogue of the Alon-Tarsi Theorem can be applied to DP-coloring.  As many researchers have noticed, for any $n \in \N$ the counterclockwise orientation of the edges of a copy of $C_{2n+2}$, $D$, has 2 even circulations and 0 odd circulations.  Moreover, the outdegree of every vertex in $D$ is 1 which means $\chi_\ell(C_{2n+2}) \leq 2$ by the Alon-Tarsi Theorem. However, $\chi_{DP}(C_{2n+2}) = 3$.  So, we are not able to replace $\chi_\ell(G)$ with $\chi_{DP}(G)$ in the statement of the Alon-Tarsi Theorem (i.e., Theorem~\ref{thm: AT}).

The purpose of this paper is to explore the possibility of applying the Combinatorial Nullstellensatz to the study of DP-coloring of graphs, despite this setback. In fact, our arguments will apply to a far-reaching generalization of many coloring problems, $S$-labeling, recently introduced by Jin, Wong, and Zhu~\cite{JW19}. Specifically, the notion of $S$-labeling is a common generalization of signed $k$-coloring, signed $\mathbb{Z}_k$-coloring, DP-coloring, group coloring, and coloring of gained graphs. Suppose that $S$ is some nonempty subset of the symmetric group $S_k$ over some set $A$ with $|A|=k$. An \emph{S-labeling} of graph $G$ is a pair $(D, \sigma)$ consisting of an orientation $D$ of $G$ and a function $\sigma : E(D) \rightarrow S$.  A \emph{proper $k$-coloring of $(D, \sigma)$} is a mapping $f: V(G) \rightarrow A$ such that $\sigma((x,y))(f(x)) \neq f(y)$ for each $(x,y) \in E(D)$. We say $G$ is \emph{$S$-$k$-colorable} if there is a proper $k$-coloring of $(D, \sigma)$ whenever $(D, \sigma)$ is an $S$-labeling of $G$.  Observe that a graph $G$ is DP-$k$-colorable if and only if $G$ is $S_k$-$k$-colorable.

\subsection{Outline of the Results}


The key obstacle to overcome in applying the Combinatorial Nullstellensatz to DP-coloring is: the graph polynomial $f_G$ is typically viewed as a polynomial over $\R$ which allows us only to prove results in the DP-coloring context on covers that correspond to list assignments. In this paper we view graph polynomials as polynomials over some appropriate finite field which allows us to apply the Combinatorial Nullstellensatz to certain covers with list sizes bounded by a power of a prime. This flexibility allows us to apply the Combinatorial Nullstellensatz, and the tools derived from it, to many covers that do not correspond to any list assignment and to coloring problems of $S$-labelings (defined above).

In Section~\ref{good}, we define the notion of \emph{good covers of prime order} which allows us to apply the Combinatorial Nullstellensatz to DP-coloring as follows (note that in the statement of the Theorem below $t$ is a prime power and every set $L(v)$ has size at most $t$).

\begin{thm} \label{thm: null}
Let $G$ be a graph with $V(G) = \{v_1, \ldots, v_n \}$, and let $\mathcal{H} = (L,H)$ be a good prime cover of $G$ of order $t$.  Suppose that $f_G \in \mathbb{F}_t[x_1, \ldots, x_n]$.  If $[ \prod_{i=1}^n x_i^{t_i}]_{f_G} \neq 0$ and $|L(v_i)| > t_i$ for each $i \in [n]$, then there exists an $\mathcal{H}$-coloring of $G$.
\end{thm}

We then show how Theorem~\ref{thm: null} can be used to prove the well-known result that non-trivial trees have DP-chromatic number equal to 2. We also discuss how this result in fact applies to coloring problems of $S$-labelings.

In Section~\ref{cone}, we show how Theorem~\ref{thm: null} above along with the Quantitative Combinatorial Nullstellensatz, a generalization of the Combinatorial Nullstellensatz, can be used to prove results on DP-coloring of cones of certain graphs.  Specifically, we prove the following result.

\begin{thm}\label{thm:bipartite}
Let $G$ be a connected bipartite graph with $|V(G)| =|E(G)|$. Then, $G'=K_1 \vee G$ is $f$-DP-colorable with $f(v_1)=2$ for the universal vertex, and $f(v)=3$ for all $v \in V(G')- \{v_1\}$.
\end{thm}

In Section~\ref{unique}, we extend some of the ideas of Section~\ref{cone} to apply to the cones of certain uniquely 3-colorable graphs. We also prove a DP-coloring analogue of the Theorem in~\cite{AM06}, where it is shown that existence of a single appropriate list assignment that leads to a unique coloring of a graph $G$, in fact implies the $f$-choosability of $G$. We illustrate applications of these results with some simple examples.

In Section~\ref{three} we utilize some properties of the finite field of order 3 to establish a sufficient algebraic condition for a graph $G$ to satisfy $\chi_{DP}(G) \leq 3$.

\begin{cor} \label{cor: suff3}
Suppose $G$ is a graph with $\chi_{DP}(G) \geq 2$ and $V(G) = \{v_1, \ldots, v_n \}$.  Let $\mathcal{F} \subseteq \mathbb{F}_3[x_1, \ldots, x_n]$ be the set of at most $2^{|E(G)|}$  polynomials given by:
$$\mathcal{F} = \left \{ \prod_{v_iv_j \in E(G), \; j>i} (x_i + b_{i,j} x_j) : \text{each}\;b_{i,j} \in \{-1, 1 \} \right \}.$$
If for each $f \in \mathcal{F}$ there exists $(t_1, t_2, \ldots, t_n) \in \prod_{i=1}^n \{0,1,2\}$ such that $[ \prod_{i=1}^n x_i^{t_i}]_f \neq 0$, then $\chi_{DP}(G) \leq 3$.
\end{cor}

We will also demonstrate that for a connected graph $G$ with cycles, we can reduce the number of polynomials to be tested to at most $2^{|E(G)|-|V(G)|+1}$.  We end Section~\ref{three} by showing some applications of Corollary~\ref{cor: suff3}. In particular, we show the converse of Corollary~\ref{cor: suff3} does not hold.  We show that $\chi_{DP}(K_{4,4} - \{e_1,e_2\})=3$, where $\{e_1, e_2\}$ is any matching of size two in $K_{4,4}$. We also completely determine the DP-chromatic number of all cycle squares: $\chi_{DP}(C_3^2)=3$, $\chi_{DP}(C_5^2) = 5$, and $\chi_{DP}(C_n^2)=4$ whenever $n \in \{4, 6, 7, 8, \ldots \}$.

While this paper initiates the study of applying the Combinatorial Nullstellensatz to DP-coloring, it is by no means exhaustive.  Much remains to be discovered regarding applications of the Combinatorial Nullstellensatz and the Quantitative Combinatorial Nullstellensatz to DP-coloring.  For example, it would be interesting to see if the results in Section~\ref{unique} could be extended to larger families of uniquely $k$-colorable graphs.  It would also be interesting to extend the results of Section~\ref{three} and find sufficient algebraic conditions for a graph to satisfy $\chi_{DP}(G) \leq t$ for $t>3$.

\section{Using the Combinatorial Nullstellensatz} \label{null}

\subsection{Combinatorial Nullstellensatz and Good Covers}\label{good}

As we will see below, applying the Combinatorial Nullstellensatz will require us to consider graph polynomials as having coefficients in some finite field.  So, when $t$ is a power of a prime, we use $\mathbb{F}_t$ to denote the finite field of order $t$.  We say that an $f$-cover $\mathcal{H} = (L,H)$ of a graph $G$ is a \emph{prime cover of $G$ of order t} whenever $t$ is a power of a prime and $\max_{v \in V(G)} f(v)\le t$. When the choice of $t$ is implicitly known, we simply use the phrase \emph{prime cover} or \emph{prime $f$-cover}. Since a prime cover could be defined with any large enough prime power $t$, unless otherwise noted, we choose $t$ to be the smallest prime power that is meaningful in context. When $f_G$ is the graph polynomial of graph $G$ and we write $f_G \in \mathbb{F}_t[x_1, \ldots, x_n]$, we are viewing $f_G$ as a polynomial in $n$ variables over $\mathbb{F}_t$.

Suppose $G$ is a graph with $V(G) = \{v_1, \ldots, v_n \}$.  From this point forward, if $\mathcal{H} = (L,H)$ is a prime cover of $G$ of order $t$, we will always name the vertices of $H$ so that $L(v) \subseteq \{(v,j) : j \in \mathbb{F}_t \}$ for each $v \in V(G)$.  Moreover, for each $v_iv_j \in E(G)$ with $j > i$, we let $A_{v_iv_j}^{\mathcal{H}}$ (resp. $B_{v_iv_j}^{\mathcal{H}}$) be the second coordinates of the set of vertices in $L(v_i)$ (resp. $L(v_j)$) saturated by the matching $E_H (L(v_i), L(v_j))$. The \emph{saturation function associated with}  $E_H (L(v_i), L(v_j))$ is the function $\sigma_{v_iv_j}^{\mathcal{H}} : A_{v_iv_j}^{\mathcal{H}} \rightarrow B_{v_iv_j}^{\mathcal{H}}$ that maps each $q \in A_{v_iv_j}^{\mathcal{H}}$ to the unique $r \in B_{v_iv_j}^{\mathcal{H}}$ with property $(v_i,q)(v_j, r) \in E_H (L(v_i), L(v_j))$.  We say that $\sigma_{v_iv_j}^{\mathcal{H}}$ is \emph{good} if there is a $\beta \in \mathbb{F}_t$ such that for each $a \in A_{v_iv_j}^{\mathcal{H}}$,
$$ a - \sigma_{v_iv_j}^{\mathcal{H}}(a) = \beta$$
where subtraction is performed in $\mathbb{F}_t$ (note: we take $\sigma_{v_iv_j}^{\mathcal{H}}$ to be \emph{good} if $A_{v_iv_j}^{\mathcal{H}} = \emptyset$); that is, there is a fixed difference between second coordinates of all pairs of matched vertices in $\mathcal{H}$ that correspond to a single edge in $G$. When $\sigma_{v_iv_j}^{\mathcal{H}}$ is not good, we say that it is \emph{bad}. It is now natural to say the cover $\mathcal{H}$ is \emph{good} if there exists a way to name the vertices of $H$ so that $L(v) \subseteq \{(v,j) : j \in \mathbb{F}_t \}$ for each $v \in V(G)$ and every associated saturation function $\sigma_{v_iv_j}^{\mathcal{H}}$, for $v_iv_j \in E(G)$ with $j > i$, is good.  From this point forward, whenever we are considering a good prime cover $\mathcal{H}=(L,H)$, we will always assume that the vertices of $H$ have been named so that every saturation function is good. There is a straightforward application of the Combinatorial Nullstellensatz when $\mathcal{H} = (L,H)$ is a good prime cover of $G$.

\begin{customthm} {\bf \ref{thm: null}}
Let $G$ be a graph with $V(G) = \{v_1, \ldots, v_n \}$, and let $\mathcal{H} = (L,H)$ be a good prime cover of $G$ of order $t$.  Suppose that $f_G \in \mathbb{F}_t[x_1, \ldots, x_n]$.  If $[ \prod_{i=1}^n x_i^{t_i}]_{f_G} \neq 0$ and $|L(v_i)| > t_i$ for each $i \in [n]$, then there is an $\mathcal{H}$-coloring of $G$.
\end{customthm}

\begin{proof}
For each $v \in V(G)$, let $P(v) = \{ j \in \mathbb{F}_t : (v,j) \in L(v) \}$.  For each $v_iv_j \in E(G)$ with $j > i$, there is a $\beta_{i,j} \in \mathbb{F}_t$ such that $a - \sigma_{v_iv_j}^{\mathcal{H}}(a) - \beta_{i,j} = 0$ for each $a \in A_{v_iv_j}^{\mathcal{H}}$ (note: we arbitrarily choose $\beta_{i,j}$ if $A_{v_iv_j}^{\mathcal{H}} = \emptyset$).  Now, let $\hat{f} \in \mathbb{F}_t[x_1, \ldots, x_n]$ be the polynomial given by
$$ \hat{f}(x_1, x_2, \ldots, x_n) = \prod_{v_iv_j \in E(G), \; j>i} (x_i-x_j - \beta_{ij}).$$
Notice that if $(p_1, p_2, \ldots, p_n) \in \prod_{i=1}^n P(v_i)$ satisfies $\hat{f}(p_1, p_2, \ldots, p_n) \neq 0$, then $I = \{(v_i, p_i) : i \in [n] \}$ is an $\mathcal{H}$-coloring of $G$.  To see why this is so, suppose for the sake of contradiction that $q > r$ and $(v_q,p_q)$ is adjacent to $(v_r, p_r)$ in $H$.  Since $\mathbb{F}_t$ is an integral domain and $\hat{f}(p_1, p_2, \ldots, p_n) \neq 0$, we know that $p_r-p_q - \beta_{r,q} \neq 0$.  On the other hand, since $(v_q,p_q)$ is adjacent to $(v_r, p_r)$ in $H$,
$$0 = p_r - \sigma_{v_rv_q}^{\mathcal{H}}(p_r) - \beta_{r,q} = p_r-p_q - \beta_{r,q}$$
which is a contradiction.

Now, notice that $f_G$ and $\hat{f}$ are polynomials of degree $\sum_{i=1}^n t_i = |E(G)|$.  Also, $[ \prod_{i=1}^n x_i^{t_i}]_{f_G} = [ \prod_{i=1}^n x_i^{t_i}]_{\hat{f}}$.  So, by the Combinatorial Nullstellensatz since $|P(v_i)| = |L(v_i)| > t_i$ for each $i \in [n]$, there is a $(p_1, p_2, \ldots, p_n) \in \prod_{i=1}^n P(v_i)$ such that $\hat{f}(p_1, p_2, \ldots, p_n) \neq 0$.  The result follows.
\end{proof}


We now comment on how Theorem~\ref{thm: null} in fact applies to $S$-$k$-coloring, the common generalization of signed $k$-coloring, signed $\mathbb{Z}_k$-coloring, DP-coloring, group coloring, and coloring of gained graphs, defined earlier. Suppose that $G$ is a graph with $V(G) = \{v_1, \ldots, v_n \}$ and $t$ is a power of a prime. Let $S_t$ be the symmetric group over $\mathbb{F}_t$. Suppose that $S$ is the subset of $S_t$ consisting of all the permutations $p \in S_t$ with the property that $i - p(i)$ is the same element of $\mathbb{F}_t$ for all $i \in \mathbb{F}_t$ (subtraction is performed in $\mathbb{F}_t$).  Theorem~\ref{thm: null} implies that if $f_G \in \mathbb{F}_t[x_1, \ldots, x_n]$, $[ \prod_{i=1}^n x_i^{t_i}]_{f_G} \neq 0$, and $t > t_i$ for each $i \in [n]$, then $G$ is $S$-$t$-colorable.  This is because given any $S$-labeling $(D, \sigma)$ of $G$ there is a corresponding $t$-fold cover $\mathcal{H} = (L,H)$ of $G$ with the following properties: (1) $\mathcal{H}$ is a good prime cover of order $t$ and (2) there is an $\mathcal{H}$-coloring of $G$ if and only if  there is a proper $t$-coloring of $(D, \sigma)$.

Theorem~\ref{thm: null} also sheds some light on why the Alon-Tarsi Theorem cannot be applied in the DP-coloring context to even cycles.  Suppose that $n \in \N$, $G = C_{2n+2}$, and $\mathcal{H}=(L,H)$ is a 2-fold cover of $G$.  If we view $f_G$ as an element of $\mathbb{F}_2[x_1, \ldots, x_n]$, then $[\prod_{i=1}^{2n+2} x_i]_{f_G} = 0$, and the hypotheses of Theorem~\ref{thm: null} will not be satisfied.  We will also see below in Section~\ref{three} that the converse of Theorem~\ref{thm: null} does not hold.

We now show how Theorem~\ref{thm: null} can be used to prove a well-known DP-coloring result.  Our proof will make use of the fact that if $\mathcal{H}=(L,H)$ is any 2-fold cover of a graph $G$ with $V(G) = \{v_1, \ldots, v_n \}$ and vertices of $H$ arbitrarily named so that $L(v) \subseteq \{(v,j) : j \in \mathbb{F}_2 \}$ for each $v \in V(G)$, then it must be that for each $v_iv_j \in E(G)$ with $j > i$, $\sigma_{v_iv_j}^{\mathcal{H}}$ is good.

\begin{pro} \label{pro: tree}
Let $T$ be a tree on at least two vertices. Then, $\chi_{DP}(T)=2$.
\end{pro}

\begin{proof}
Let $V(T) = \{v_1, \ldots, v_n \}$. Since $T$ has at least one edge, $\chi_{DP}(T) > 1$.  Suppose $\mathcal{H} = (L,H)$ is an arbitrary 2-fold cover of $G$.  Clearly, $\mathcal{H}$ is a good prime cover of $G$.  It is well known that there exists an acyclic orientation, $D$, of $T$ with the property that $d_D^+ (v_i) \leq 1$.  Thus, $\text{diff}(D)=1$.  Working over $\mathbb{F}_2$, this means that $[ \prod_{i=1}^n x_i^{d_D^+ (v_i)}]_{f_G} \neq 0$.  Theorem~\ref{thm: null} then implies that $G$ is $\mathcal{H}$-colorable.  It follows that $\chi_{DP}(T) \leq 2$.
\end{proof}

Even though the converse of Theorem~\ref{thm: null} does not hold in general, when the hypotheses of Theorem~\ref{thm: null} are not met, Theorem~\ref{thm: null} can provide a ``clue" as to how one might construct a cover to establish a lower bound on the DP-chromatic number of a graph. Here is an illustration of this idea.

Suppose that $G$ is the line graph of a copy of $K_{p+1}$ where $p$ is an odd prime.  Suppose $V(G) = \{v_1, \ldots, v_n \}$ and we view $f_G$ as a polynomial in $n$ variables over $\mathbb{F}_p$. In~\cite{S14}, Schauz showed that $[ \prod_{i=1}^n x_i^{p-1}]_{f_G} = 0$.  This means that there \emph{may} exist a $p$-fold cover $\mathcal{H}$ of $G$ such that there is no $\mathcal{H}$-coloring of $G$ and for each $v_iv_j \in E(G)$ with $j > i$, $\sigma_{v_iv_j}^{\mathcal{H}}$ is good. This may be seen as clue. Bernshteyn and Kostochka~\cite{BK17} famously showed that if $M$ is the line graph of a $d$-regular graph, then $\chi_{DP}(M) \geq d+1$. In order to prove this result, they construct a $d$-fold cover $\mathcal{H}'$ of $M$ for which there is no $\mathcal{H}'$-coloring, and in their proof, it is easy to observe that $\mathcal{H}'$ is a good prime cover of order $d$, thus confirming the clue for constructing such a cover. Of course, in practice, using such a clue may not be easy. We will give a more detailed use of similar ideas in Proposition~\ref{pro: appcyclesquare} below.

\subsection{DP-Coloring the Cone of a Graph}\label{cone}

In 2008 Schauz gave a generalization of the Combinatorial Nullstellensatz (see~\cite{S12, S14}).  The generalization is called Quantitative Combinatorial Nullstellensatz, and we will show how this generalization along with Theorem~\ref{thm: null} can be applied to DP-coloring of the cone of certain graphs. The \emph{cone of a graph} $G$ is defined to be $K_1 \vee G$, and the vertex corresponding to the $K_1$ is called the \emph{universal vertex}. In this section, we will always denote the universal vertex with $v_1$.

\begin{thm} [Quantitative Combinatorial Nullstellensatz~\cite{S12, S14}] \label{thm: quantnull}
Let $P_1, P_2, \ldots, P_n$ be finite nonempty subsets of a field $\mathbb{F}$ and $\mathcal{P} = \prod_{i=1}^n P_i$.  Suppose for $j \in [n]$, $d_j = |P_j|-1$.  For every $P(x_1, \ldots, x_n) \in \mathbb{F}[x_1, \ldots, x_n]$ of degree at most $\sum_{j=1}^n d_j$,
$$ \left [ \prod_{i=1}^n x_i^{d_i} \right]_P = \sum_{(p_1, p_2, \ldots, p_n) \in \mathcal{P}} N(p_1, p_2, \ldots, p_n)^{-1} P(p_1, p_2, \ldots, p_n)$$
where
$$N(p_1, p_2, \ldots, p_n) = \prod_{j=1}^n \left(\prod_{\epsilon \in P_j - \{p_j \}} (p_j - \epsilon) \right).$$
Note:  If $|P_j|=1$, we take $\prod_{\epsilon \in P_j - \{p_j \}} (p_j - \epsilon)$ to equal 1.
\end{thm}

\begin{pro} \label{pro: bipartite1}
Let $G$ be a connected bipartite graph with $|V(G)| =|E(G)|$. Let $G' = K_1 \vee G$.
Then, there is an $\mathcal{H}$-coloring of $G'$ for every good prime $f$-cover $\mathcal{H}$ with $f(v_1)=1$ for the universal vertex, and $f(v)=3$ for all $v \in V(G)$.
\end{pro}

\begin{proof}
Let $G$ have the bipartition $X,Y$ with $|X|=m$, $|Y|=n$, and let $X= \{v_2, \ldots, v_{m+1} \}$ and $Y= \{v_{m+2}, \ldots, v_{m+n+1} \}$.
Suppose we view the graph polynomial $f_{G'}$ as a polynomial in $m+n+1$ variables over $\mathbb{F}_3$.  Suppose that $P_1 = \{0 \}$ and $P_i = \{0,1,2\}$ for $i=2,3, \ldots,  m+n+1$.  Let $\mathcal{P} = \prod_{i=1}^n P_i$.  Let $\textbf{y}=(y_1, y_2, \ldots, y_{m+n+1})$ be the element of $\mathcal{P}$ such that $y_1=0$, $y_2=y_3= \cdots = y_{m+1} = 1$, and $y_{m+2}=y_{m+3}= \cdots = y_{m+n+1} = 2$.  Similarly, let $\textbf{z}=(z_1, z_2, \ldots, z_{m+n+1})$ be the element of $\mathcal{P}$ such that $z_1=0$, $z_2=z_3= \cdots = z_{m+1} = 2$, and $z_{m+2}=z_{m+3}= \cdots = z_{m+n+1} = 1$.  Since $G$ has a unique bipartition, $\textbf{y}$ and $\textbf{z}$ are the only elements in $\mathcal{P}$ for which $f$ is nonzero.

Using the notation of Theorem~\ref{thm: quantnull}, we see that $N(\textbf{y})=N(\textbf{z})=(-1)^{m+n}$.  So, Theorem~\ref{thm: quantnull} implies
\begin{align*}
\left [ \prod_{i=2}^{m+n+1} x_i^{2} \right]_{f_{G'}} &= \sum_{(p_1, p_2, \ldots, p_{n+m+1}) \in \mathcal{P}} N(p_1, p_2, \ldots, p_{n+m+1})^{-1} f(p_1, p_2, \ldots, p_{n+m+1}) \\
&= N(\textbf{y})^{-1} f(\textbf{y}) + N(\textbf{z})^{-1} f(\textbf{z}) \\
&= (-1)^{m+n} (-1)^m (-1)^{m+n} + (-1)^{m+n} (-1)^n = 2 (-1)^m.
\end{align*}

So, we have that $\left [ \prod_{i=2}^{m+n+1} x_i^{2} \right]_{f_{G'}} \neq 0$.  Theorem~\ref{thm: null} then implies that there is an $\mathcal{H}$-coloring of $G'$.
\end{proof}

It is worth mentioning that the result of Proposition~\ref{pro: bipartite1} could also be deduced from the weaker result that there is an $\mathcal{H}$-coloring of $G' = K_1 \vee C_{2k+2}$ for every good prime $f$-cover $\mathcal{H}$ with $f(v_1)=1$ for the universal vertex, and $f(v)=3$ for all $v \in V(G') - \{v_1 \}$.  This is because when $G$ is a connected bipartite graph with $|V(G)| =|E(G)|$ and $G'=K_1 \vee G$, deleting all the vertices in $V(G')$ of degree less than 3 results in a copy of $K_1 \vee C_{2k+2}$ for some $k \in \N$.

With this in mind, we will now work toward an application of Proposition~\ref{pro: bipartite1}.  However, before we present the application we need some terminology and a result from~\cite{KM19}.  Suppose $G$ is a graph and $\mathcal{H} = (L,H)$ is an $m$-fold cover of $G$ with $m \geq \chi(G)$.  We say there is a \emph{natural bijection between the $\mathcal{H}$-colorings of $G$ and the proper $m$-colorings of $G$} if for any set $S$ of size $m$ and for each $v \in V(G)$ it is possible to set $L(v) = \{(v,j) : j \in S \}$ so that whenever $uv \in E(G)$, $(u,j)$ and $(v,j)$ are adjacent in $H$ for each $j \in S$.  To see that this definition makes sense, suppose there is a natural bijection between the $\mathcal{H}$-colorings of $G$ and the proper $m$-colorings of $G$.  Then, for each $v \in V(G)$ it is possible to set $L(v) = \{(v,j) : j \in [m] \}$ so that whenever $uv \in E(G)$, $(u,j)$ and $(v,j)$ are adjacent in $H$ for each $j \in [m]$.  Note that if $\mathcal{I}$ is the set of $\mathcal{H}$-colorings of $G$ and $\mathcal{C}$ is the set of proper $m$-colorings of $G$, then the function $f: \mathcal{C} \rightarrow \mathcal{I}$ given by
$$f(c) = \{ (v, c(v)) : v \in V(G) \}$$
is a bijection.

\begin{pro} [\cite{KM19}] \label{pro: treeDP}
Let $T$ be a tree on $n$ vertices and $\mathcal{H} = (L,H)$ be an $m$-fold cover of $T$ such that $m \geq 2$ and $E_H(L(u),L(v))$ is a perfect matching whenever $uv \in E(T)$.  Then, there is a natural bijection between the $\mathcal{H}$-colorings of $T$ and the proper $m$-colorings of $T$.
\end{pro}

Suppose $G$ is a graph, and let $M = K_1 \vee G$. Suppose that $v_1$ is the universal vertex of $M$, and suppose that $\mathcal{H} = (L,H)$ is a cover of $M$.  Also, suppose that for fixed $l,m \in \N$ satisfying $l \leq m$, $L(v_1) = \{(v_1,j) : j \in \{0, \ldots, l-1 \} \}$ and $L(v) = \{(v,j) : j \in \{0, \ldots, m-1 \} \}$ for each $v \in V(G)$.  We refer to the edges of $H$ connecting distinct parts of the partition $\{L(v) : v \in V(M) \}$ as \emph{cross-edges}.  For each $j \in \{0, \ldots, l-1 \}$ and $v \in V(G)$, let
$$H^{(j)} = H - N_H[(v_1,j)] \; \; \text{and} \; \; L^{(j)}(v) = L(v) -  N_H((v_1,j)).$$
Now, suppose for each $v \in V(G)$, $|E_H(L(v_1),L(v))|=l$ (i.e., $E_H(L(v_1),L(v))$ is as large as possible).  Then, $\mathcal{H}^{(j)} = (L^{(j)},H^{(j)})$ is an $(m-1)$-fold cover of $G$.  We say that $(v_1,t) \in L(v_1)$ is a \emph{level vertex} if $H^{(t)}$ contains precisely $|E(G)|(m-1)$ cross-edges (i.e., $H^{(t)}$ contains the maximum possible number of cross-edges).

With this terminology and Proposition~\ref{pro: treeDP} in mind, we are ready to present an application of Proposition~\ref{pro: bipartite1}.

\begin{customthm} {\bf \ref{thm:bipartite}}
Let $G$ be a connected bipartite graph with $|V(G)| =|E(G)|$. Then, $G'=K_1 \vee G$ is $f$-DP-colorable with $f(v_1)=2$ for the universal vertex, and $f(v)=3$ for all $v \in V(G')- \{v_1\}$.
\end{customthm}

As discussed following the proof of Proposition~\ref{pro: bipartite1}, when $G$ is a connected bipartite graph with $|V(G)| =|E(G)|$ and $G'=K_1 \vee G$, deleting all the vertices in $V(G')$ of degree less than 3 results in a copy of $K_1 \vee C_{2k+2}$ for some $k \in \N$. Hence, to complete the proof of Theorem~\ref{thm:bipartite}, it is enough to prove the following statement.

\begin{pro} \label{pro: bipartite2}
Let $G' = K_1 \vee C_{2k+2}$. Then, $G'$ is $f$-DP-colorable with $f(v_1)=2$ for the universal vertex, and $f(v)=3$ for all $v \in V(G')- \{v_1\}$.
\end{pro}

\begin{proof}
We label the vertices of the copy of $C_{2k+2}$ used to form $G'$ (in cyclic order) as: $v_2, v_3, \ldots, v_{2k+3}$.
For the sake of contradiction, suppose that $\mathcal{H} = (L,H)$ is a prime $f$-cover of $G'$ such that there is no $\mathcal{H}$-coloring of $G'$.  We may suppose that for each $v_iv_j \in E(G')$, $|E_H(L(v_i),L(v_j))|$ is as large as possible.  Let $\mathcal{H}' = (L',H')$ be the 3-fold cover for the path, $P$, with vertices $v_2, v_3, \ldots, v_{2k+3}$, given by $H' = H[ \bigcup_{j=2}^{2k+3} L(v_j) ]-E_H(L(v_2),L(v_{2k+3}))$ and $L'(v_j) = L(v_j)$ for each $j \in \{2,3, \ldots, 2k+3 \}$.  For each $j \in \{2,3, \ldots, 2k+3 \}$ suppose we name the vertices in $L(v_j)$: $(v_j,0), (v_j,1), (v_j,2)$, so that there is a natural bijection between the proper 3-colorings of $P$ and the $\mathcal{H}'$-colorings of $P$ (this is possible by Proposition~\ref{pro: treeDP}).  Also, arbitrarily name the vertices in $L(v_1)$: $(v_1,0), (v_1,1)$.  Now, we have named all the vertices in $V(H)$.

We may assume that $\sigma_{v_2v_{2k+3}}^{\mathcal{H}}$ is bad since otherwise we could delete an element from $L(v_1)$ to obtain a cover of $G'$ satisfying the hypotheses of Proposition~\ref{pro: bipartite1}.  So, we assume without loss of generality that $\{ (v_2,0)(v_{2k+3},0), (v_2,1)(v_{2k+3},2), (v_2,2)(v_{2k+3},1) \} \subset E(H)$.  It is now easy to see that at most one vertex in $L(v_1)$ can be a level vertex.  Suppose that $(v_1,0)$ is not a level vertex, and for $j \in \{2,3, \ldots, 2k+3 \}$ let $L''(v_j) = L(v_j) - N_H((v_1,0))$.  Also, let $H'' = H[ \bigcup_{j=2}^{2k+3} L''(v_j) ]$.  Notice that $\mathcal{H}'' = (L'',H'')$ is a cover for the copy of $C_{2k+2}$, $C$, used to form $G'$.  Since $(v_1,0)$ is not a level vertex, there exists $r, q \in \{2,3, \ldots, 2k+3 \}$ such that $v_rv_q \in E(G')$ and $|E_{H''}(L''(v_r), L''(v_q))| \leq 1$.  So, there is a $(v_q, i) \in L''(v_q)$ that is not adjacent in $H$ to any elements of $L''(v_r)$.  It is therefore possible to greedily construct an $\mathcal{H}''$-coloring $I$ of $C$ that contains $(v_q,i)$.  Then, $I \cup \{(v_1,0) \}$ is an $\mathcal{H}$-coloring of $G$ which is a contradiction.
\end{proof}


\subsection{Uniquely Colorable Graphs}\label{unique}

We will now prove a generalization of Proposition~\ref{pro: bipartite1}.  We begin with some terminology.  A graph $G$ is said to be \emph{uniquely $k$-colorable} if there is only one partition of its vertex set into $k$ color classes.  From this point forward if $G$ is a uniquely $k$-colorable graph, we always let $\{I_1, I_2, \ldots, I_k \}$ be the unique partition of $V(G)$ into $k$ color classes.  Also, for each $i \in [k]$, we let $n_i = |I_i|$, and for each $1 \leq i < j \leq k$, we let $m_{i,j} = |E_G(I_i, I_j)|$.

\begin{pro} \label{pro: bipartite3}
Suppose $G$ is a uniquely 3-colorable graph with $2(n_1 + n_2 + n_3) = |E(G)|$, $n_2 + m_{1,3} \equiv 0$, $n_3 + m_{1,2} \equiv 1$, and $n_1 + m_{2,3} \equiv 2 \;(\mod \; 3)$.  Let $G' = K_1 \vee G$. Then, there is an $\mathcal{H}$-coloring of $G'$ for every good prime $f$-cover $\mathcal{H}$ of order 4 with $f(v_1)=1$ for the universal vertex, and $f(v)=4$ for all $v \in V(G)$.
\end{pro}

\begin{proof}
We view $f_{G'}$ as a polynomial in $n_1+n_2+n_3+1$ variables over $\mathbb{F}_4 = \{0,1,x,x+1\}$.  Clearly, $f_{G'}$ is a polynomial of degree at most $|E(G')|=3(n_1+n_2+n_3)$.  Suppose that $P_1 = \{0 \}$ and $P_i = \{0, 1, x, x+1\}$ for $i=2,3, \ldots,  n_1+n_2+n_3+1$.  Let $\mathcal{P} = \prod_{i=1}^n P_i$.  Let $I_1= \{v_2, \ldots, v_{n_1+1} \}$, $I_2= \{v_{n_1+2}, \ldots, v_{n_1+n_2+1} \}$, and $I_3 = \{v_{n_1+n_2+2}, \ldots, v_{n_1 + n_2 + n_3 + 1}\}$, be the three unique color classes of $G$.

For each $i \in [6]$ let $f_i : \{1,2,3 \} \rightarrow \{1, x, x+1 \}$ be the bijective function such that: $f_1(1)=1$, $f_1(2) = x$, $f_1(3) = x+1$, $f_2(1)=1$, $f_2(2) = x+1$, $f_2(3) = x$, $f_3(1)=x$, $f_3(2) = 1$, $f_3(3) = x+1$, $f_4(1)=x$, $f_4(2) = x+1$, $f_4(3) = 1$, $f_5(1)=x+1$, $f_5(2) = 1$, $f_5(3) = x$, $f_6(1)=x+1$, $f_6(2) = x$, and $f_6(3) = 1$.  Now, for each $i \in [6]$, let $\textbf{y}_i=(y_1, y_2, \ldots, y_{n_1+n_2+n_3+1})$ be the element of $\mathcal{P}$ such that $y_1=0$, $y_2=y_3= \cdots = y_{n_1+1} = f_i(1)$, $y_{n_1+2}=y_{n_1+3}= \cdots = y_{n_1+n_2+1} = f_i(2)$, and $y_{n_1+n_2+2}=y_{n_1+n_2+3}= \cdots = y_{n_1+n_2+n_3+1} = f_i(3)$.  Since $G$ is uniquely 3-colorable, $\textbf{y}_1, \textbf{y}_2, \ldots, \textbf{y}_6$ are the only elements in $\mathcal{P}$ for which $f_{G'}$ is nonzero.

Using the notation of Theorem~\ref{thm: quantnull}, we see that for each $i \in [6]$, $N(\textbf{y}_i)=(-1)^{n_1+n_2+n_3} = 1$.  So, Theorem~\ref{thm: quantnull} and the fact that $n_2 + m_{1,3} \equiv 0$, $n_3 + m_{1,2} \equiv 1$, and $n_1 + m_{2,3} \equiv 2 \; (\mod \; 3)$ implies
\begin{align*}
&\left [ \prod_{i=2}^{n_1+n_2+n_3+1} x_i^{3} \right]_{f_{G'}} \\
&= \sum_{(p_1, p_2, \ldots, p_{n_1+n_2+n_3+1}) \in \mathcal{P}} N(p_1, p_2, \ldots, p_{n_1+n_2+n_3+1})^{-1} f_{G'}(p_1, p_2, \ldots, p_{n_1+n_2+n_3+1}) \\
&= \sum_{i=1}^6 f_{G'}(\textbf{y}_i)\\
&= x^{n_2 + m_{1,3}}(x+1)^{n_3+m_{1,2}} + x^{n_3 + m_{1,2}}(x+1)^{n_2+m_{1,3}} + x^{n_1 + m_{2,3}}(x+1)^{n_3+m_{1,2}} \\
&+ x^{n_1 + m_{2,3}}(x+1)^{n_2+m_{1,3}} + x^{n_3 + m_{1,2}}(x+1)^{n_1+m_{2,3}} + x^{n_2 + m_{1,3}}(x+1)^{n_1+m_{2,3}} \\
&= (1)(x+1) + x(1) + (x+1)^2 + (x+1)(1) + x^2 + (1)(x) \\
&= x + x+1 = 1.
\end{align*}

So, we have that $\left [ \prod_{i=2}^{n_1+n_2+n_3+1} x_i^{3} \right]_{f_{G'}} \neq 0$.  Theorem~\ref{thm: null} then implies that there is an $\mathcal{H}$-coloring of $G'$.
\end{proof}

For a simple application of Proposition~\ref{pro: bipartite3}, consider $G = \overline{K_2} \vee P_5$. In $G$, let $v_1, v_2$ be the vertices of the copy of $\overline{K_2}$, and let the vertices of the copy of $P_5$ (in order) be: $u_1, u_2, u_3, u_4, u_5$.  Clearly, $G$ is uniquely 3-colorable, and if we let $I_1 = \{u_2, u_4 \}$, $I_2 = \{v_1, v_2 \}$, and $I_3 = \{u_1, u_3, u_5 \}$, then it is clear that the hypotheses of Proposition~\ref{pro: bipartite3} are satisfied.

We will now present a DP-coloring analogue of a result that appears in~\cite{AM06}, where it is shown that existence of an appropriate list assignment that leads to a unique coloring of $G$, implies $f$-choosability of $G$ for $\sum_{v \in V(G)} f(v) = |V(G)| + |E(G)|$ (for a nontrivial application of this result in~\cite{AM06} see~\cite{KM17}).

\begin{pro} \label{pro: unique}
Let $G$ be a graph, and let $f: V(G) \rightarrow \N$ be such that $\sum_{v \in V(G)} f(v) = |V(G)| + |E(G)|$.
If $P$ is a list assignment of $G$ such that for each $v \in V(G)$, $|P(v)|=f(v)$ and $P(v) \subseteq \mathbb{F}_t$ for some prime power $t$, and there is a unique proper $P$-coloring of $G$, then there is an $\mathcal{H}$-coloring of $G$ for every good prime $f$-cover $\mathcal{H}$ of order $t$.
\end{pro}

\begin{proof}
 Suppose $V(G) = \{v_1, \ldots, v_n \}$ and $\mathcal{H}$ is a good prime $f$-cover of $G$ of order $t$.  We view $f_G$ as a polynomial in $n$ variables over $\mathbb{F}_t$.  Clearly, $f_{G}$ is a polynomial of degree at most $|E(G)| = \sum_{i=1}^n (|P(v_i)|-1)$.  For $i \in [n]$, let $d_i = |P(v_i)|-1$.  Let $\mathcal{P} = \prod_{i=1}^n P_i$.  Since there is a unique proper $P$-coloring of $G$, we know there is exactly one element in $\mathcal{P}$ at which $f_G$ is nonzero.  Suppose $(a_1, \ldots, a_n) \in \mathcal{P}$ and $f_G(a_1, \ldots, a_n) \neq 0$.  Using the notation of Theorem~\ref{thm: quantnull},
$$ \left[ \prod_{i=1}^n x_i^{d_i} \right]_{f_G} = N(a_1, \ldots, a_n)^{-1} f_G(a_1, \ldots, a_n).$$
Since $\mathbb{F}_t$ is an integral domain, we have that $ N(a_1, \ldots, a_n)^{-1} f_G(a_1, \ldots, a_n) \neq 0$.  Theorem~\ref{thm: null} then implies that there is an $\mathcal{H}$-coloring of $G$.
\end{proof}

We can use Proposition~\ref{pro: unique} to prove a slightly stronger (albeit obvious~\footnote{Corollary~\ref{cor: tree} can be easily proven with greedy coloring.  Our purpose in the proof of Corollary~\ref{cor: tree} is to give a very simple application of Proposition~\ref{pro: unique}.}) version of Proposition~\ref{pro: tree}.  The result makes use of the fact that if $T$ is a tree, there is an ordering of the elements of $V(T)$, $v_1, \ldots, v_n$, so that for each $i \in \{2, \ldots, n \}$, $v_i$ has exactly one neighbor among $v_1, \ldots, v_{i-1}$.

\begin{cor} \label{cor: tree}
Let $T$ be a tree on $n$ vertices with at least one edge. Suppose $v_1, \ldots, v_n$ is an ordering of the elements of $V(T)$ so that for each $i \in \{2, \ldots, n \}$, $v_i$ has exactly one neighbor among $v_1, \ldots, v_{i-1}$.  Then, $T$ is $f$-DP-colorable for $f(v_1)=1$ and $f(v_i)=2$ for each $i \in \{2, \ldots, n \}$.
\end{cor}

\begin{proof}
Suppose that $\mathcal{H}$ is an arbitrary $f$-cover of $T$.  We know that $\mathcal{H}$ is a good prime $f$-cover of order 2.  Clearly, $\sum_{v \in V(T)} f(v) = 2n-1 = |V(T)| + |E(T)|$.  Let $P$ be the list assignment for $T$ given by $L(v_1) = \{0 \}$ and $L(v_i) = \{0,1\}$ for each $i \in \{2, \ldots, n \}$.  Then, it is clear that there is a unique proper $P$-coloring of $T$.  Proposition~\ref{pro: unique} then implies that there is an $\mathcal{H}$-coloring of $T$.
\end{proof}

\section{Prime Covers of Order 3} \label{three}

As we have seen above, prime covers of order 2 are ideal covers to be working with when applying Theorem~\ref{thm: null} (since when we have a prime cover of order 2, $\mathcal{H}$, $\sigma_{v_iv_j}^{\mathcal{H}}$ is always good, as we saw in Proposition~\ref{pro: tree} and Corollary~\ref{cor: tree}).  In this section, we will see that if we work a bit harder, we can obtain a Theorem that applies to all prime covers of order 3.  Throughout this section, unless otherwise noted, we assume that $G$ is a graph with $V(G) = \{v_1, \ldots, v_n \}$.  We also assume that $\mathcal{H} = (L,H)$ is a prime cover of $G$ of order 3.

\subsection{Dealing with Bad Matchings}

In the context of prime covers of order 3, we have the following observation.

\begin{obs} \label{obs: crucial}
Suppose $\sigma_{v_iv_j}^{\mathcal{H}}$ is bad.  Then there exists $\beta_{i,j} \in \mathbb{F}_3$ such that $a + \sigma_{v_iv_j}^{\mathcal{H}}(a) = \beta_{i,j}$ for each $a \in A_{v_iv_j}^{\mathcal{H}}$.
\end{obs}

Let $\mathcal{M} = \{\sigma_{v_iv_j}^{\mathcal{H}} : v_iv_j \in E(G), j > i \}$.  Let $B: \mathcal{M} \rightarrow \{-1,1\}$, be the function given by
\[ B(\sigma_{v_iv_j}^{\mathcal{H}}) = \begin{cases}
      -1 & \text{if $\sigma_{v_iv_j}^{\mathcal{H}}$ is good} \\
      1 & \text{if $\sigma_{v_iv_j}^{\mathcal{H}}$ is bad.}
   \end{cases}
\]

\begin{thm} \label{thm: null3}
Suppose $G$ is a graph with $V(G) = \{v_1, \ldots, v_n \}$ and $\mathcal{H} = (L,H)$ is a prime cover of $G$ of order $3$ with the vertices of $H$ arbitrarily named so that $L(v) \subseteq \{(v,j) : j \in \mathbb{F}_3 \}$ for each $v \in V(G)$.   Let $f(x_1, \ldots, x_n) \in \mathbb{F}_3[x_1, \ldots, x_n]$ be given by
$$f(x_1, \ldots, x_n) = \prod_{v_iv_j \in E(G), \; j>i} (x_i + B(\sigma_{v_iv_j}^{\mathcal{H}}) x_j).$$
If $[ \prod_{i=1}^n x_i^{t_i}]_f \neq 0$ and $|L(v_i)| > t_i$ for each $i \in [n]$, then there is an $\mathcal{H}$-coloring of $G$.
\end{thm}

\begin{proof}
For each $v \in V(G)$, let $P(v) = \{ j \in \mathbb{F}_3 : (v,j) \in L(v) \}$.  By Observation~\ref{obs: crucial}, for each $v_iv_j \in E(G)$ with $j > i$, there is a $\beta_{i,j} \in \mathbb{F}_3$ such that $a + B(\sigma_{v_iv_j}^{\mathcal{H}}) \sigma_{v_iv_j}^{\mathcal{H}}(a) - \beta_{i,j} = 0$ for each $a \in A_{v_iv_j}^{\mathcal{H}}$ (note: we arbitrarily choose $\beta_{i,j}$ if $A_{v_iv_j}^{\mathcal{H}} = \emptyset$).  Now, let $\hat{f} \in \mathbb{F}_3[x_1, \ldots, x_n]$ be the polynomial given by:
$$ \hat{f}(x_1, x_2, \ldots, x_n) = \prod_{v_iv_j \in E(G), \; j>i} (x_i + B(\sigma_{v_iv_j}^{\mathcal{H}})x_j - \beta_{ij}).$$
Similar to the analogous result proven in the proof of Theorem~\ref{thm: null}, if $(p_1, p_2, \ldots, p_n) \in \prod_{i=1}^n P(v_i)$ satisfies $\hat{f}(p_1, p_2, \ldots, p_n) \neq 0$, then $I = \{(v_i, p_i) : i \in [n] \}$ is an $\mathcal{H}$-coloring of $G$.

Notice that $f$ and $\hat{f}$ are polynomials of degree $\sum_{i=1}^n t_i = |E(G)|$.  Also, $[ \prod_{i=1}^n x_i^{t_i}]_f = [ \prod_{i=1}^n x_i^{t_i}]_{\hat{f}}$.  So, by the Combinatorial Nullstellensatz since $|P(v_i)| = |L(v_i)| > t_i$ for each $i \in [n]$, there is a $(p_1, p_2, \ldots, p_n) \in \prod_{i=1}^n P(v_i)$ such that $\hat{f}(p_1, p_2, \ldots, p_n) \neq 0$.  The result follows.
\end{proof}

Interestingly, the polynomials $f$ and $\hat{f}$ in the proof of Theorem~\ref{thm: null3} are considered in~\cite{WQ19} where the authors study the Alon-Tarsi Number and Modulo Alon-Tarsi Number of signed graphs.  We can use Theorem~\ref{thm: null3} to get a sufficient algebraic condition for a graph $G$ to satisfy $\chi_{DP}(G) \leq 3$.

\begin{customcor} {\bf \ref{cor: suff3}}
Suppose $G$ is a graph with $\chi_{DP}(G) \geq 2$ and $V(G) = \{v_1, \ldots, v_n \}$.  Let $\mathcal{F} \subseteq \mathbb{F}_3[x_1, \ldots, x_n]$ be the set of at most $2^{|E(G)|}$  polynomials given by:
$$\mathcal{F} = \left \{ \prod_{v_iv_j \in E(G), \; j>i} (x_i + b_{i,j} x_j) : \text{each}\;b_{i,j} \in \{-1, 1 \} \right \}.$$
If for each $f \in \mathcal{F}$ there exists $(t_1, t_2, \ldots, t_n) \in \prod_{i=1}^n \{0,1,2\}$ such that $[ \prod_{i=1}^n x_i^{t_i}]_f \neq 0$, then $\chi_{DP}(G) \leq 3$.
\end{customcor}

\begin{proof}
Suppose that $\mathcal{H} = (L,H)$ is an arbitrary 3-fold cover of $G$ with the vertices of $H$ arbitrarily named so that $L(v) \subseteq \{(v,j) : j \in \mathbb{F}_3 \}$ for each $v \in V(G)$.  Let $f(x_1, x_2, \ldots, x_n) \in \mathbb{F}_3[x_1, \ldots, x_n]$ be the polynomial given by
$$f(x_1, x_2, \ldots, x_n) = \prod_{v_iv_j \in E(G), \; j>i} (x_i + B(\sigma_{v_iv_j}^{\mathcal{H}}) x_j).$$
Clearly, $f \in \mathcal{F}$.  So, there exists $(t_1, t_2, \ldots, t_n) \in \prod_{i=1}^n \{0,1,2\}$ such that $[ \prod_{i=1}^n x_i^{t_i}]_f \neq 0$.  Since $|L(v_i)| = 3 > t_i$ for each $i \in [n]$, there is an $\mathcal{H}$-coloring of $G$ by Theorem~\ref{thm: null3}.  Since $\mathcal{H}$ was arbitrary, we have that $\chi_{DP}(G) \leq 3$.
\end{proof}


Using the idea of natural bijections (as discussed in the previous Section), it is possible to simplify Corollary~\ref{cor: suff3} a bit by reducing the number of polynomials that need to be checked.

\begin{cor} \label{cor: suff32}
Suppose $G$ is a connected graph containing a cycle with $V(G) = \{v_1, \ldots, v_n \}$, and $T$ is a spanning tree of $G$.  Let $\mathcal{F}  \subseteq \mathbb{F}_3[x_1, \ldots, x_n]$ be the set of at most $2^{|E(G)|-|V(G)|+1}$  polynomials given by:
$$\mathcal{F} = \left \{ \left [ \prod_{v_iv_j \in E(T), \; j>i} (x_i - x_j) \right] \left[\prod_{v_iv_j \in E(G)-E(T), \; j>i} (x_i + b_{i,j} x_j) \right] : \text{each}\;b_{i,j} \in \{-1, 1 \} \right \}.$$
If for each $f \in \mathcal{F}$ there exists $(t_1, t_2, \ldots, t_n) \in \prod_{i=1}^n \{0,1,2\}$ such that $[ \prod_{i=1}^n x_i^{t_i}]_f \neq 0$, then $\chi_{DP}(G) \leq 3$.
\end{cor}

\begin{proof}
Suppose that $\mathcal{H} = (L,H)$ is an arbitrary 3-fold cover of $G$.  We may assume that $|E_H(L(v_i),L(v_j))|=3$ whenever $v_iv_j \in E(G)$.  Let $H' = H - \bigcup_{v_iv_j \in E(G)-E(T)} E_H(L(v_i),L(v_j))$ and  $\mathcal{H}' = (L,H')$.  Notice that $\mathcal{H}'$ is a 3-fold cover of $T$.  By Proposition~\ref{pro: treeDP}, it is possible to assign names to the vertices of $H$ such that $L(v) = \{(v,j) : j \in \mathbb{F}_3 \}$ for each $v \in V(G)$ and there is a natural bijection between the proper $3$-colorings of $T$ and the $\mathcal{H}'$-colorings of $T$.   Let $f(x_1, \ldots, x_n)  \in \mathbb{F}_3[x_1, \ldots, x_n]$ be the polynomial given by
$$f(x_1, \ldots, x_n) = \prod_{v_iv_j \in E(G), \; j>i} (x_i + B(\sigma_{v_iv_j}^{\mathcal{H}}) x_j).$$
Since $\sigma_{v_iv_j}^{\mathcal{H}}$ is good whenever $v_iv_j \in E(T)$, $f \in \mathcal{F}$.  So, there exists $(t_1, t_2, \ldots, t_n) \in \prod_{i=1}^n \{0,1,2\}$ such that $[ \prod_{i=1}^n x_i^{t_i}]_f \neq 0$.  Since $|L(v_i)| = 3 > t_i$ for each $i \in [n]$, there is an $\mathcal{H}$-coloring of $G$ by Theorem~\ref{thm: null3}.  Since $\mathcal{H}$ was arbitrary, we have that $\chi_{DP}(G) \leq 3$.
\end{proof}

\subsection{Some Applications}

\begin{pro} \label{pro: appbipartite}
$\chi_{DP}(K_{4,4} - \{e_1,e_2\})=3$, where $\{e_1, e_2\}$ is a matching of size two in $K_{4,4}$.
\end{pro}

\begin{proof}
Since $G'= K_{4,4} - \{e_1,e_2\}$ contains a copy of $C_4$, we have that $\chi_{DP}(G') > 2$.  Suppose $V(G') = \{v_1, \ldots, v_8 \}$.  Using a computer to analyze the 16384 polynomials over $\mathbb{F}_3$ contained in
$$\mathcal{F} = \left \{ \prod_{v_iv_j \in E(G'), \; j>i} (x_i + b_{i,j} x_j) : \text{each}\;b_{i,j} \in \{-1, 1 \} \right \},$$
we see that the hypotheses of Corollary~\ref{cor: suff3} are satisfied.  Thus, $\chi_{DP}(G') = 3$.
\end{proof}

The converse of Corollary~\ref{cor: suff3} does not hold.  We now present an example to show this.  It is known that $\chi_{DP}(K_{3,t})=3$ if and only if $t=2,3,4,5$ (\cite{M18}).  Suppose that $G$ is a copy of $K_{3,5}$ with partite sets $\{v_1, v_2, v_3\}$ and $\{v_4, v_5, v_6, v_7, v_8\}$.  View the polynomial $f_G$ as an element of $\mathbb{F}_3[x_1, \ldots, x_8]$.  We know $f_G$ is homogenous of degree 15.  Let $T$ be the set containing the 8 elements in $\prod_{i=1}^8 \{0,1,2\}$ with coordinates summing to 15.  It is easy to compute that for each $(t_1, t_2, \ldots, t_8) \in T$, $[ \prod_{i=1}^8 x_i^{t_i}]_{f_G} = 0$.  Notice that this example also shows that the converses of Theorems~\ref{thm: null} and~\ref{thm: null3} do not hold.

Even though the converse of Corollary~\ref{cor: suff3} does not hold, when the hypotheses of Corollary~\ref{cor: suff3} are not satisfied, we can still use that information to give ourselves ``clues" about how one might construct a 3-fold cover for which there is no coloring.

As an example suppose that $G = C_6^2$ where the vertices of $G$ in cyclic order are: $v_1, v_2, \ldots, v_6$.  Consider the following polynomials over $\mathbb{F}_3$:
\begin{align*}
&f_1(x_1, x_2, \ldots, x_6) = f_G(x_1, x_2, \ldots, x_6) \; \; \text{and} \\
&f_2(x_1, x_2, \ldots, x_6) = (x_1+x_2)(x_1+x_3) \prod_{v_iv_j \in E(G)-\{v_1v_2, v_1v_3\}, \; j>i} (x_i - x_j).
\end{align*}
It is not hard to compute: $[ \prod_{i=1}^6 x_i^{2}]_{f_1} = 0$ and $[ \prod_{i=1}^6 x_i^{2}]_{f_2} = 1$.  So, Theorem~\ref{thm: null3} tells us that if $\mathcal{H} = (L, H)$ is a 3-fold cover for $G$ with vertices of $H$ named so that $\sigma_{v_iv_j}^{\mathcal{H}}$ is bad when $(i,j)=(1,2)$ and when $(i,j)=(1,3)$ and good otherwise, then there is an $\mathcal{H}$-coloring of $G$.  On the other hand, there \emph{may} exist a 3-fold cover for $G$, $\mathcal{H}$, that is a good prime cover of order 3 for which there is no $\mathcal{H}$-coloring of $G$.

This second observation gives us an idea about what the proof of the following Proposition might look like.

\begin{pro} \label{pro: appcyclesquare}
$\chi_{DP}(C_{3k}^2) > 3$ for any $k \geq 2$.
\end{pro}

\begin{proof}
Let $G = C_{3k}^2$, and denote the vertices of $G$ in cyclic order as: $v_1, v_2, \ldots, v_{3k}$.  We will prove the desired result by constructing a 3-fold cover of $G$ for which there is no coloring.  For each $i \in [3k]$, let $L(v_i) = \{(v_i,t) : t \in \mathbb{F}_3 \}$.  Then, let $H$ be the graph with $V(H) = \bigcup_{i=1}^{3k} L(v_i)$.  We draw the edges of $H$ as follows: draw edges so that $H[L(v_i)]$ is a complete graph for each $i \in [3k]$ and for each $v_iv_j \in E(G) - \{v_{3k-2}v_{3k}, v_{3k-1}v_{3k} \}$ draw an edge between $(v_i, t)$ and $(v_j,t)$ for each $t \in \mathbb{F}_3$.  Finally, for $j = 3k-2, 3k-1$ and each $t \in \mathbb{F}_3$, draw an edge between $(v_j, t)$ and $(v_{3k}, t+1)$ where addition is performed in $\mathbb{F}_3$.

Now, $\mathcal{H} = (L,H)$ is clearly a 3-fold cover of $G$.  We claim that there is no $\mathcal{H}$-coloring of $G$.  For the sake of contradiction, suppose that $I$ is an $\mathcal{H}$-coloring of $G$.  We know that there exist $r, q \in \mathbb{F}_3$ with $r \neq q$ such that $(v_1, q), (v_2,r) \in I$.  Suppose that $l$ is the element in $\mathbb{F}_3 - \{q, r \}$.  The edges in $E_H(L(v_1), L(v_{3k})) \cup E_H(L(v_2), L(v_{3k}))$ imply that $(v_{3k},l) \in I$.  By observing the edges in the induced subgraph $H[\bigcup_{i=1}^{3k-1} L(v_i)]$, we see that $(v_{3k-1},q), (v_{3k-2},r) \in I$.  It must be that $l = q+1$ or $l = r+1$.  So, $(v_{3k},l)$ is adjacent to $(v_{3k-1},q)$ or $(v_{3k-2},r)$.  This contradicts the fact that $I$ is an $\mathcal{H}$-coloring of $G$.
\end{proof}

Note that in the proof above $\mathcal{H} = (L,H)$ is a 3-fold cover of $G$ such that $\sigma_{v_iv_j}^{\mathcal{H}}$ is good for all $v_iv_j \in E(G)$.  It is also worth mentioning that $\chi_{\ell}(C_{3k}^2)=3$ can be proven via the Alon-Tarsi Theorem~\cite{JM98}.  It is also known that $\chi_{\ell}(C_{n}^2) = 4$ whenever $n \geq 6$ and $n$ is not divisible by 3~\cite{PW03}.  This along with Proposition~\ref{pro: appcyclesquare} and the fact that $\chi_{DP}(G) \leq \Delta(G)$ provided that $G$ is neither a cycle nor a complete graph, immediately allows us to give the DP-chromatic number of all cycle squares.

\begin{cor} \label{cor: cyclesquare}
$\chi_{DP}(C_3^2)=3$, $\chi_{DP}(C_5^2) = 5$, and $\chi_{DP}(C_n^2)=4$ whenever $n \in \{4, 6, 7, 8, \ldots \}$.
\end{cor}

\noindent {\bf Acknowledgment.} We are thankful to Xuding Zhu for his encouragement and helpful comments.

\end{document}